\documentclass{amsart}
\usepackage{amsmath}
\usepackage{amssymb}
\usepackage{amsthm}
\usepackage{enumerate}
\usepackage[dvips]{graphicx}
\theoremstyle{definition}
\newtheorem{definition}{Definition}[section]
\theoremstyle{plain}
\newtheorem{lemma}[definition]{Lemma}
\newtheorem{theorem}[definition]{Theorem}
\newtheorem{proposition}[definition]{Proposition}
\newtheorem{corollary}[definition]{Corollary}
\theoremstyle{remark}
\newtheorem{remark}[definition]{Remark}

\newcommand{\myds}{\operatorname{D}_\Sigma}
\newcommand{\myint}{\operatorname{int}}

\makeatletter
\@namedef{subjclassname@2010}{
  \textup{2010} Mathematics Subject Classification}
\makeatother

\begin{document}

\title[Dimension inequality for uniformly locally o-minimal structure]{Dimension inequality for a definably complete uniformly locally o-minimal structure of the second kind}
\author[M. Fujita]{Masato Fujita}
\address{Department of Liberal Arts,
Japan Coast Guard Academy,
5-1 Wakaba-cho, Kure, Hiroshima 737-8512, Japan}
\email{fujita.masato.p34@kyoto-u.jp}

\begin{abstract}
Consider a definably complete uniformly locally o-minimal expansion of the second kind of a densely linearly ordered abelian group.
Let $f:X \rightarrow R^n$ be a definable map, where $X$ is a definable set and $R$ is the universe of the structure.
We demonstrate the inequality $\dim(f(X)) \leq \dim(X)$ in this paper.
As a corollary, we get that the set of the points at which $f$ is discontinuous is of dimension smaller than $\dim(X)$.
We also show that the structure is defiably Baire in the course of the proof of the inequality.
\end{abstract}



\maketitle

\section{Introduction}\label{sec:intro}

A uniformly locally o-minimal structure of the second kind was first defined and investigated in the author's previous work \cite{Fuji}.
It enjoys several tame properties such as local monotonicity.
In addition, it admits local definable cell decomposition when it is definably complete.

In \cite{Fuji}, the author defined dimension of a set definable in a locally o-minimal structure admitting local definable cell decomposition.
Many assertions on dimension known in o-minimal structures \cite{vdD} also hold true for locally o-minimal structures admitting local definable cell decomposition which are not necessarily definably complete \cite[Section 5.5]{Fuji}.
An exception is the inequality $\dim(f(X)) \leq \dim(X)$, where $f:X \rightarrow R^n$ is a definable map.
The author gave an example which does not satisfy the above dimension inequality in \cite[Remark 5.5]{Fuji}.
The structure in the example is not definably complete.
A question is whether the dimension inequality holds true when the structure is definably complete.
This paper gives an affirmative answer to this question.
Our main theorem is as follows:

\begin{theorem}\label{thm:main}
Let $\mathcal R=(R,<,+,0,\ldots)$ be a definably complete uniformly locally o-minimal expansion of the second kind of a densely linearly ordered abelian group.
The inequality 
\[
\dim(f(X)) \leq \dim(X)
\]
 holds true for any definable map $f:X \rightarrow R^n$.
\end{theorem}

We get the following corollary:
\begin{corollary}\label{cor:discont}
Let $\mathcal R=(R,<,+,0,\ldots)$ be the same structure as Theorem \ref{thm:main}.
Let $f:X \rightarrow R$ be a definable function.
The set of the points at which $f$ is discontinuous is of dimension smaller than $\dim(X)$.
\end{corollary}

The author proved the dimension inequality in \cite[Theorem 2.4]{Fuji2} when the universe of the structure is the set of reals.
This fact is not a direct corollary of the above theorem because the structure should be an expansion of an abelian group in the theorem.

The paper is organized as follows.
In Section \ref{sec:preliminaries}, we first review definitions used in the paper.
We prove several basic facts in Section \ref{sec:basic}.
Satisfaction of the dimension inequality is relevant to defiably Baire property introduced in \cite{FS}.
Section \ref{sec:baire} treats the definably Baire property.
We show that a definably complete uniformly locally o-minimal expansion of the second kind of a densely linearly ordered abelian group is definably Baire in the section.
We finally demonstrate Theorem \ref{thm:main} in Section \ref{sec:proof}.

We introduce the terms and notations used in this paper.
The term `definable' means `definable in the given structure with parameters' in this paper.
A \textit{CBD} set is a closed, bounded and definable set.
For any set $X \subset R^{m+n}$ definable in a structure $\mathcal R=(R,\ldots)$ and for any $x \in R^m$, the notation $X_x$ denotes the fiber defined as $\{y \in R^n\;|\; (x,y) \in X\}$.
For a linearly ordered structure $\mathcal R=(R,<,\ldots)$, an open interval is a definable set of the form $\{x \in R\;|\; a < x < b\}$ for some $a,b \in R$.
It is denoted by $(a,b)$ in this paper.
We define a closed interval in the same manner and it is denoted by $[a,b]$.
An open box in $R^n$ is the direct product of $n$ open intervals.
A closed box is defined similarly.
Let $A$ be a subset of a topological space.
The notations $\myint(A)$ and $\overline{A}$ denote the interior and the closure of the set $A$, respectively.
The notation $|S|$ denotes the cardinality of a set $S$.

\section{Definitions}\label{sec:preliminaries}

We review the definitions given in the previous works.
The definition of a definably complete structure is found in \cite{M} and \cite{DMS}.
A locally o-minimal structure is defined and investigated in \cite{TV}.
Readers can find the definitions of uniformly locally o-minimal structures of the second kind and locally o-minimal structures admitting local definable cell decomposition in \cite{Fuji}.
We use $\myds$-sets introduced in \cite{DMS}.

\begin{definition}[$\myds$-sets]
Consider an expansion of a linearly ordered structure $\mathcal R=(R,<,0,\ldots)$.
A \textit{parameterized family} of definable sets is the family of the fibers of a definable set.
A parameterized family $\{X_{r,s}\}_{r>0,s>0}$ of CBD subsets of $R^n$ is called a \textit{$\myds$-family} if $X_{r,s} \subset X_{r',s}$ and  $X_{r,s'} \subset X_{r,s}$ whenever $r \leq r'$ and $s \leq s'$.
A definable subset $X$ of $R^n$ is a \textit{$\myds$-set} if $X = \displaystyle\bigcup_{r>0,s>0} X_{r,s}$ for some $\myds$-family $\{X_{r,s}\}_{r>0,s>0}$.

A parameterized family of definable sets $\{X_{s}\}_{s>0}$  is a \textit{definable decreasing family of CBD sets} if we have $X_s=X_{r,s}$ for some $\myds$-family $\{X_{r,s}\}_{r>0,s>0}$ with $X_{r_1,s}=X_{r_2,s}$ for all $r_1$, $r_2$ and $s$.
\end{definition}

We next review definably Baire property introduced in \cite{FS}.
\begin{definition}
Consider an expansion of a densely linearly ordered structure.
A parameterized family of definable sets $\{X_{r}\}_{r>0}$  is called a \textit{definable increasing family} if $X_r \subset X_{r'}$ whenever $0<r<r'$.
A definably complete expansion of a densely linearly ordered structure is \textit{definably Baire} if the union $\bigcup_{r>0} X_r$ of any definable increasing family $\{X_r\}_{r>0}$ with $\myint\left(\overline{X_r}\right)=\emptyset$ has an empty interior.
\end{definition}

The following proposition is a direct corollary of the local definable cell decomposition theorem \cite[Theorem 4.2]{Fuji}. 

\begin{proposition}\label{prop:baire_basic}
Consider a definably complete uniformly locally o-minimal structure of the second kind.
It is definably Baire if and only if the union $\bigcup_{r>0} X_r$ of any definable increasing family $\{X_r\}_{r>0}$ with $\myint(X_r)=\emptyset$ has an empty interior.
\end{proposition}
\begin{proof}
Because $\myint\left(\overline{X_r}\right)\not=\emptyset$ iff $\myint(X_r) \not=\emptyset$ iff $X_r$ contains an open cell in this case by \cite[Theorem 4.2]{Fuji}. 
\end{proof}

The dimension of a set definable in a locally o-minimal structure admitting local definable cell decomposition is defined in \cite[Section 5]{Fuji}.
We get the following lemma on the dimension of the projection image.
A lemma similar to it is found in \cite{Fuji2}, but we give a complete proof here.

\begin{lemma}\label{lem:dim_pre}
Consider a locally o-minimal structure $\mathcal R=(R,<,\ldots)$ admitting local definable cell decomposition.
Let $X$ be a definable subset of $R^{m+n}$ and $\pi: R^{m+n} \rightarrow R^m$ be a coordinate projection.
Assume that  the fibers $X_x$ are of dimension $\leq 0$ for all $x \in R^m$.
Then, we have $\dim X \leq \dim \pi(X)$.
\end{lemma}
\begin{proof}
For any $(a,b) \in R^m \times R^n$, there exist open boxes $B_a \subset R^m$ and $B_b \subset  R^n$ with $(a,b) \in B_a \times B_b$ and $\dim (X \cap (B_a \times B_b)) = \dim \pi(X \cap (B_a \times B_b))$ by \cite[Lemma 5.4]{Fuji}.
We have $\dim \pi(X \cap (B_a \times B_b)) \leq \dim \pi(X)$ by \cite[Lemma 5.1]{Fuji}.
On the other hand, we have $\dim (X) = \displaystyle\sup_{(a,b) \in R^m \times R^n} \dim (X \cap (B_a \times B_b))$ by \cite[Corollary 5.3]{Fuji}.
We have finished the proof.
\end{proof}

\section{Preliminaries}\label{sec:basic}
From now on, we consider a definably complete uniformly locally o-minimal expansion of the second kind of a densely linearly ordered abelian group $\mathcal R=(R,<,+,0,\ldots)$.
We demonstrate several basic facts in this section.

\begin{lemma}\label{lem:decreasing}
Let $X$ be a bounded definable set.
There exists a definable decreasing family of CBD sets $\{X_s\}_{s>0}$ with $X=\bigcup_{s>0}X_s$.
\end{lemma}
\begin{proof}
We demonstrate the lemma by the induction on $d=\dim(X)$.
When $d=0$, $X$ is discrete and closed by \cite[Corollary 5.3]{Fuji}.
We have only to set $X_s=X$ for all $s>0$ in this case.

We next consider the case in which $d>0$.
Let $\partial X$ denote the frontier of $X$.
We have $\dim\overline{\partial X}<d$ by \cite[Theorem 5.6]{Fuji}.
We get $\dim(X \cap \overline{\partial X}) < d$ by \cite[Proposition 5.1]{Fuji}.
There exists a definable decreasing family of CBD sets $\{Y_s\}_{s>0}$ with $X \cap \overline{\partial X}=\bigcup_{s>0}Y_s$ by the induction hypothesis.
Set $Z_s=\{x \in \overline{X}\;|\; d(x, \overline{\partial X}) \geq s\}$ for all $s>0$, where the notation $d(x, \overline{\partial X})$ denotes the distance of the point $x$ to the set $\overline{\partial X}$.
They are CBD.
It is obvious that $\bigcup_{s>0}Z_s = \overline{X} \setminus \overline{\partial X} = X \setminus \overline{\partial X}$.
Set $X_s=Y_s \cup Z_s$.
The family $\{X_s\}_{s>0}$ is a definable decreasing family we are looking for.
\end{proof}

\begin{lemma}\label{lem:decreasing2}
Any definable set $X$ is a $\myds$-set.
That is, there exists a $\myds$-family $\{X_{r,s}\}_{r>0,s>0}$ with $X=\bigcup_{r>0,s>0}X_{r,s}$.
\end{lemma}
\begin{proof}
Let $X$ be a definable subset of $R^n$.
Set $X_r=X \cap [-r,r]^n$.
We can construct subsets $X_{r,s}$ of $X_r$ satisfying the condition in the same manner as the proof of Lemma \ref{lem:decreasing2}.
We omit the details.
\end{proof}

\begin{lemma}\label{lem:basic1}
Let $X$ be a bounded definable set and $\{X_s\}_{s>0}$ be a definable decreasing family of CBD sets with $X=\bigcup_{s>0}X_s$.
The CBD set $X_s$ has a nonempty interior for some $s>0$ if $X$ has a nonempty interior.
\end{lemma}
\begin{proof}
We prove the lemma following the same strategy as the proof of \cite[3.1]{DMS}.
Let $X$ be a definable subset of $R^n$.
We prove the lemma by the induction on $n$.
We first consider the case in which $n=1$.
Assume that $\myint(X_s)=\emptyset$ for all $s>0$.
Fix an arbitrary point $a \in R$.
There exist a positive integer $N$, an interval $I$ with $a \in I$ and $t>0$ such that, for any $0<s<t$, $I \cap X_s$ contains an open interval or consists of at most $N$ points by \cite[Theorem 4.2]{Fuji}. 
The sets $I \cap X_s$ consist of at most $N$ points because $\myint(X_s)=\emptyset$.
We get $|X \cap I|=\left|\bigcup_{s>0}(I \cap X_s)\right| \leq N$.
In particular, $X$ has an empty interior.

We next consider the case in which $n>1$.
Assume that $X$ has a nonempty interior.
We show that the definable set $X_s$ has a nonempty interior for some $s>0$.
A closed box $B=C \times I \subset R^{n-1} \times R$ is contained in $X$.
We have $B=\bigcup_{s>0}(B \cap X_s)$.
Hence, we may assume that $X$ is a closed box $B$ without loss of generality.

Shrinking $B$ if necessary, we may assume that the fiber $(X_s)_x$ consists of at most $M$ points and $N$ closed intervals for some $M>0$, $N>0$ and any sufficiently small $s>0$ and $x \in C$ by \cite[Theorem 4.2]{Fuji}.
Set $I=[c_1,c_2]$.
Take $2N$ distinct points in the open interval $(c_1,c_2)$, say $b_1, \ldots, b_{2N}$.
We may assume that $b_i < b_j$ whenever $i < j$.
Set $b_0=c_1$ and $b_{2N+1}=c_2$.
Put $I_j=[b_{j-1},b_j]$ for all $1 \leq j \leq 2N+1$.

Consider the sets $Y^k_s=\{x \in C\;|\; I_k \subset (X_s)_x\}$ for all $s>0$ and $1 \leq k \leq 2N+1$.
They are CBD.
Therefore, $\left\{ \bigcup_{k=1}^{2N+1} Y^k_s \right\}_{s>0}$ is a definable decreasing family of CBD sets.
We demonstrate that $C=\bigcup_{s}\bigcup_{k=1}^{2N+1} Y^k_s$.
Let $x \in C$ be fixed.
We have only to show that $I_k \subset (X_s)_x$ for some $k$ and $s$.
For any $k$, there exists $s_k>0$ such that $\myint(I_k \cap (X_{s_k})_x) \not= \emptyset$ by the induction hypothesis because $\{I_k \cap (X_s)_x\}_{s>0}$ is a decreasing family of CBD sets with $I_k = \bigcup_{s} I_k \cap (X_s)_x$. 
Take $s=\min\{s_k\;|\;1 \leq k \leq 2N+1\}$.
We have $\myint(I_k \cap (X_s)_x) \not=\emptyset$ for all $1 \leq k \leq 2N+1$.
Assume that $I_k \not\subset (X_s)_x$ for all $k$.
A maximal closed interval in $(X_s)_x$ should be contained in $I_k$, $I_k \cup I_{k+1}$ or $I_{k-1} \cup I_k$ for some $k$.
Therefore, $\myint(I_j \cap (X_s)_x)$ is empty for some $1 \leq j \leq 2N+1$.
Contradiction.
We have proven that $I_k \subset (X_s)_x$ for some $k$ and $s$.

Apply the induction hypothesis to $C=\bigcup_{s>0}\bigcup_{k=1}^{2N+1} Y^k_s$.
The set $\bigcup_{k=1}^{2N+1} Y^k_s$ has a nonempty interior for some $s>0$.
The CBD set $Y^k_s$ has a nonempty interior for some $k$ by \cite[Theorem 3.3]{Fuji}.
The CBD set $X_s$ has a nonempty interior because $I_k \times Y^k_s$ is contained in $X_s$.
\end{proof}

\begin{lemma}\label{lem:basic2}
Assume that $\mathcal R$ is definably Baire.
Let $X$ be a definable set and $\{X_{r,s}\}_{r>0,s>0}$ be a $\myds$-family with $X=\bigcup_{r>0,s>0}X_{r,s}$.
The CBD set $X_{r,s}$ has a nonempty interior for some $r>0$ and $s>0$ if $X$ has a nonempty interior.
\end{lemma}
\begin{proof}
Let $X$ be a definable subset of $R^n$.
Set $X'_{r,s}=X_{r,s} \cap [-r,r]^n$.
We have $X=\bigcup_{r>0,s>0}X_{r,s}'$.
We may assume that $X_r=\bigcup_{r>0,s>0}X_{r,s}$ is bounded considering $X'_{r,s}$ instead of $X_{r,s}$.
The lemma is now immediate from Proposition \ref{prop:baire_basic} and Lemma \ref{lem:basic1}.
\end{proof}

\section{On definably Baire property}\label{sec:baire}
 We demonstrate that the structure $\mathcal R$ is definably Baire.
 
\begin{lemma}\label{lem:interval}
Let $X$ be a bounded definable subset of $R^{n+1}$.
Set 
\[
S=\{x \in R^n\;|\; X_x \text{ contains an open interval}\}\text{.}
\]
The set $S$ has an empty interior if $X$ has an empty interior.
\end{lemma}
\begin{proof}
Assume that $S$ has a nonempty interior.
There exists a definable decreasing family of CBD sets $\{X_s\}_{s>0}$ with $X=\bigcup_{s>0}X_s$ by Lemma \ref{lem:decreasing}.
Set $S_s=\{x \in R^n\;|\; \exists t \in R,\ [t-s,t+s] \subset (X_s)_x\}$ for all $s>0$.
They are CBD  by \cite[Lemma 1.7]{M} because they are the projection images of the CBD sets $S_s=\{(x,t) \in R^n \times R\;|\;  [t-s,t+s] \subset (X_s)_x\}$.
We have $S=\bigcup_{s>0}S_s$.
In fact, it is obvious that $\bigcup_{s>0}S_s \subset S$ by the definition.
Take a point $x \in S$.
There exist $t \in R$ and $s_1>0$ with $[t-s_1,t+s_1] \subset X_x$.
In particular, we have $\myint(X_x) \not= \emptyset$.
We have $\myint(X_{s_2})_x \not= \emptyset$ for some $s_2>0$ by Lemma \ref{lem:basic1}.
We may assume that $[t-s_1,t+s_1] \subset (X_{s_2})_x$ by taking new $s_1$ and $t$ again. 
Set $s=\min\{s_1,s_2\}$, then we have $x \in S_s$.
We have demonstrated that $S=\bigcup_{s>0}S_s$.

Again by Lemma \ref{lem:basic1}, we have $\myint(S_s) \not= \emptyset$ for some $s>0$.
We obtain $\myint(X_s) \not= \emptyset$ by \cite[2.8(2)]{DMS}.
We get $\myint(X) \not= \emptyset$.
\end{proof}

We reduce to the one-dimensional case.

\begin{lemma}\label{lem:easy_case}
The structure $\mathcal R$ is definably Baire if the union $\bigcup_{r>0}S_r$ of any definable increasing family $\{S_r\}_{r>0}$ of subsets of $R$ has an empty interior whenever $S_r$ have empty interiors for all $r>0$. 
\end{lemma}
\begin{proof}
Let $\{X_r\}_{r>0}$ be a definable increasing family of subsets of $R^n$.
Set $X=\bigcup_{r>0}X_r$.
We have only to show that the definable set $X_r$ has a nonempty interior for some $r>0$ if $X$ has a nonempty interior.
The definable set $X$ contains a bounded open box $B$.
We may assume that $X$ is a bounded open box $B$ without loss of generality by considering $B$ and $\{X_r \cap B\}_{r>0}$ in place of $X$ and $\{X_r\}_{r>0}$, respectively. 

We prove the lemma by the induction on $n$.
The lemma is obvious when $n=0$.
We next consider the case in which $n>0$.
We lead to a contradiction assuming that $X_r$ have empty interiors for all $r>0$.
Let $\pi:R^n \rightarrow R^{n-1}$ be the projection forgetting the last coordinate.
We have $B=B_1 \times I$ for some open box $B_1$ in $R^{n-1}$ and some open interval $I$.
Consider the set $Y_r=\{x \in B_1\;|\;\ \text{the fiber }(X_r)_x \text{ contains an open interval}\}$ for all $r>0$.
They have empty interiors by Lemma \ref{lem:interval}. 
The union $\bigcup_{r>0} Y_r$ has an empty interior by the induction hypothesis.
In particular, we have $B_1 \not=\bigcup_{r>0} Y_r$ and we can take a point $x \in B_1 \setminus \left(\bigcup_{r>0} Y_r\right)$.
Since $x \not\in \bigcup_{r>0} Y_r$, the fiber $(X_r)_x$ does not contain an open interval for any $r>0$.
Therefore, the union $\bigcup_{r>0}(X_r)_x$ has an empty interior by the assumption.
On the other hand, we have $I=\bigcup_{r>0}(X_r)_x$ because $B=\bigcup_{r>0}X_r$.
It is a contradiction.
\end{proof}

We prove that $\mathcal R$ is definably Baire now. 
\begin{theorem}\label{thm:baire}
A definably complete uniformly locally o-minimal expansion of the second kind of a densely linearly ordered abelian group is definably Baire.
\end{theorem}
\begin{proof}
Let $\mathcal R=(R,<,+,0,\ldots)$ be the considered structure.
Let $\{X_r\}_{r>0}$ be a definable increasing family of subsets of $R$.
Set $X=\bigcup_{r>0}X_r$.
We have only to show that the definable set $X$ has an empty interior if $X_r$ have empty interiors for all $r>0$ by Lemma \ref{lem:easy_case}.
Note that $X_r$ are discrete and closed because the structure is locally o-minimal.

Assume that $X$ has a nonempty interior.
The definable set $X$ contains an open interval.
Take a point $a$ contained in the open interval.
Consider the definable function $f:\{r \in R\;|\;r >0\} \rightarrow \{x \in R\;|\; x  > a\}$ defined by $f(r)= \inf \{x >a \;|\; x \in X_r\}$.
It is obvious that $f$ is a decreasing function because $\{X_r\}_{r>0}$ is a definable increasing family. 
We demonstrate that $\lim_{r \to \infty}f(r)=a$. 
Let $b$ be an arbitrary point sufficiently close to $a$ with $b>a$.
Since $X=\bigcup_{r>0}X_r$ contains a neighborhood of $a$, there exists a positive element $r  \in R$ with $b \in X_r$.
We have $a<f(r) \leq b$ by the definition of $f$.
We have shown that $\lim_{r \to \infty}f(r)=a$. 

Consider the image $\operatorname{Im}(f)$ of the function $f$.
Take a sufficiently small open interval $I$ containing the point $a$ with $I \subset X$.
The intersection $I \cap \operatorname{Im}(f)$ is a finite union of points and open intervals because it is definable in the locally o-minimal structure $\mathcal R$.
Take an arbitrary point $b \in \operatorname{Im}(f)$ and a point $r>0$ with $b=f(r)$.
Since $X_r$ is closed, we have $b \in X_r$.
Any point $b' \in \operatorname{Im}(f)$ with $b' > b$ is also contained in $X_r$.
In fact, take a point $r'>0$ with $b'=f(r')$.
If $r'>r$, the set $X_{r'}$ contains the point $b$ because $X_r \subset X_{r'}$.
We have $b'=f(r') \leq b$ by the definition of the function $f$.
It is a contradiction.
If $r'<r$, we have $b' \in X_{r'} \subset X_r$.

Set $b_1 = \inf\{b' \in \operatorname{Im}(f)\;|\;b'>b\} $.
We have $b_1 \in X_r$ and $b_1> b$ because $\{b' \in \operatorname{Im}(f)\;|\;b'>b\} \subset X_r$ and $X_r$ is closed and discrete.
The open interval $(b,b_1)$ has an empty intersection with $\operatorname{Im}(f)$.
We have shown that $I \cap \operatorname{Im}(f)$ does not contain an open interval.
The set $I \cap \operatorname{Im}(f)$ consists of finite points.
It is a contradiction to the fact that $\lim_{r \to \infty}f(r)=a$. 
\end{proof}

\begin{remark}
It is already known that a definably complete expansion of an ordered field is definably Baire \cite{H}.
Our research target is a uniformly locally o-minimal structure of the second kind.
A uniformly locally o-minimal expansion of the second kind of an ordered field is o-minimal by \cite[Proposition 2.1]{Fuji}.
In this case, it is trivially definably Baire by the definable cell decomposition theorem \cite[Chapter 3, (2,11)]{vdD}.
We have more interest in the case in which the structure is not an expansion of an ordered field.
\end{remark}

\section{Proof of Theorem \ref{thm:main}}\label{sec:proof}

We demonstrate Theorem \ref{thm:main} in this section.
We first show that a definable map is continuous on an open subset of the domain of definition.

\begin{lemma}\label{lem:cont_function}
A definable map $f:U \rightarrow R^n$ defined on an open set $U$ is continuous on a nonempty definable open subset of $U$.
\end{lemma}
\begin{proof}
The structure $\mathcal R$ is definably Baire by Theorem \ref{thm:baire}.
We may use Lemma \ref{lem:basic2} in the proof.

Let $U$ be a definable open subset of $R^m$.
Consider the projection $\pi:R^{m+n} \rightarrow R^m$ onto the first $m$ coordinates.
The notation $\Gamma(f)$ denotes the graph of $f$.
There exists a $\myds$-family $\{X_{r,s}\}_{r,s}$ with $\Gamma(f)=\bigcup_{r,s} X_{r,s}$ by Lemma \ref{lem:decreasing2}. 
Note that $\pi(X_{r,s})$ is CBD by \cite[Lemma 1.7]{M}.
We have $U=\bigcup_{r,s}\pi(X_{r,s})$ and the fiber $\pi^{-1}(x) \cap \Gamma(f)$ is a singleton for any $x \in U$.
Therefore, we obtain $X_{r,s}=\Gamma(f|_{\pi(X_{r,s})})$, where $f|_{\pi(X_{r,s})}$ is the restriction of $f$ to $\pi(X_{r,s})$.
Take a closed box $B$ contained in $U$.
The family $\{\pi(X_{r,s}) \cap B\}$ is a $\myds$-family and $B=\bigcup_{r,s} \pi(X_{r,s}) \cap B$.
The CBD set $\pi(X_{r,s}) \cap B$ has a nonempty interior for some $r$ and $s$ by Lemma \ref{lem:basic2}.
Take a closed box $B'$ contained in $\pi(X_{r,s}) \cap B$.
The set $X_{r,s} \cap (B' \times R^n)=\Gamma(f|_{B'})$ is closed.
Therefore, $f$ is continuous on $\myint(B')$.
\end{proof}

We finally prove Theorem \ref{thm:main}.
\begin{proof}[Proof of Theorem \ref{thm:main}]
We prove the following assertion:
\begin{description}
\item[$(*)$] The inequality $\dim(f(X)) \leq \dim(X)$ holds true for any definable map $f:X \rightarrow R^n$.
\end{description}
Lemma \ref{lem:basic2} is available as in the proof of Lemma \ref{lem:cont_function} for the same reason.

Set $d=\dim(f(X))$.
We demonstrate that $\dim(X) \geq d$.
We can reduce to the case in which the image $f(X)$ is an open box $B$ of dimension $d$.
In fact, there exist an open box $B$ in $R^{d}$ and a definable map $g:B \rightarrow f(X)$ such that the map $g$ is a definable homeomorphism onto its image by the definition of dimension \cite[Definition 5.1]{Fuji}.
Set $Y=f^{-1}(g(B))$ and $h=g^{-1} \circ f|_Y:Y \rightarrow B$.
When $\dim(Y) \geq d$, we get $\dim(X) \geq d$ by \cite[Lemma 5.1]{Fuji} because $Y$ is a subset of $X$. 
We may assume that $f(X)=B$ by considering $Y$ and $h$ instead of $X$ and $f$, respectively.

We next reduce to the case in which the map $f$ is the restriction of a coordinate projection.
Consider the graph $G = \Gamma(f) \subset R^{m+d}$ of the definable map $f$.
Let $\pi:R^{m+d} \rightarrow R^{d}$ be the projection onto the last $d$ coordinates.
We have $\dim(X) \geq \dim (G) \geq d$ by Lemma \ref{lem:dim_pre} when $\dim(G) \geq d$.
We may assume that $f:X \rightarrow B$ is the restriction of the projection $\pi:R^{m+d} \rightarrow R^{d}$ to $X$.

We have a $\myds$-family $\{X_{r,s}\}_{r>0,s>0}$ with $X=\bigcup_{r,s} X_{r,s}$ by Lemma \ref{lem:decreasing2}. 
The family $\{f(X_{r,s})\}_{r>0,s>0}$ is also a $\myds$-family by \cite[Lemma 1.7]{M} because $f$ is the restriction of a projection.
We have $B=\bigcup_{r,s}f(X_{r,s})$.
The CBD set $f(X_{r,s})$ has a nonempty interior for some $r>0$ and $s>0$ by Lemma \ref{lem:basic2}.
We fix such $r>0$ and $s>0$.
Take an open box $U$ contained in $f(X_{r,s})$.
Note that the inverse image $\{y \in X_{r,s}\;|\;f(y)=x\}$ of $x \in U$ is CBD because $f$ is continuous. 
Consider a definable function $\varphi:U \rightarrow X_{r,s}$ given by $\varphi(x)=\operatorname{\mathbf{lexmin}}\{y \in X_{r,s}\;|\;f(y)=x\}$, where the notation $\operatorname{\mathbf{lexmin}}$ denotes the lexicographic minimum defined in \cite{M}.
We can get an open box $V$ contained in $U$ such that the restriction $\varphi|_V$ of $\varphi$ to $V$ is continuous by Lemma \ref{lem:cont_function}.
The definable set $X_{r,s}$ is of dimension $ \geq d$ by the definition of dimension because it contains the graph of the definable continuous map $\varphi|_V$ defined on the open box $V$ in $R^{d}$.
We have $\dim X \geq \dim (X_{r,s})  \geq d$ by \cite[Lemma 5.1]{Fuji}.
We have proven Theorem \ref{thm:main}.
\end{proof}

The proof of Corollary \ref{cor:discont} is the same as that of \cite[Corollary 2.6]{Fuji2}.
We give a proof here because it is brief.
\begin{proof}[Proof of Corollary \ref{cor:discont}]
Let $\mathcal D$ be the set of points at which the definable function $f$ is discontinuous. 
Assume that the domain of definition $X$ is a definable subset of $R^m$.
Let $G$ be the graph of $f$.
We have $\dim(G)=\dim(X)$ by Lemma \ref{lem:dim_pre} and Theorem \ref{thm:main}.
Set $\mathcal E=\{(x,y) \in X \times  R\;|\; y=f(x) \text{ and } f \text{ is discontinuous at }x\}$.
We get $\dim(\mathcal E) < \dim(G)$ by \cite[Theorem 4.2, Corollary 5.3]{Fuji}.
Let $\pi: R^{m+1} \rightarrow  R^m$ be the projection forgetting the last coordinate.
We have $\mathcal D = \pi(\mathcal E)$ by the definitions of $\mathcal D$ and $\mathcal E$.
We finally obtain $\dim(\mathcal D) = \dim(\pi(\mathcal E))  \leq \dim(\mathcal E) < \dim(G)=\dim(X)$ by Theorem \ref{thm:main}.
\end{proof}

\end{document}